\documentclass[a4paper,12pt]{amsart}
\usepackage{amsmath,amsthm,amsfonts,amssymb}
\usepackage[all]{xy}
\newdir{ >}{{}*!/-5pt/@{>}}
\DeclareMathOperator{\colim}{Colim}
\DeclareMathOperator{\im}{Im}
\DeclareMathOperator{\coker}{Coker}
\DeclareMathOperator{\id}{Id}
\DeclareMathOperator{\ID}{ID}

\DeclareMathOperator{\Hom}{Hom}

\DeclareMathOperator{\rank}{Rank}
\newtheorem{theorem}{Theorem}[section]
\newtheorem*{theorem*}{Theorem}
\newtheorem{proposition}[theorem]{Proposition}
\newtheorem{conjecture}[theorem]{Conjecture}
\newtheorem{lemma}[theorem]{Lemma}

\theoremstyle{definition}
\newtheorem{example}[theorem]{Example}
\newtheorem{definition}[theorem]{Definition}
\newtheorem{remark}[theorem]{Remark}

\title{Injective dimension of sheaves of rational vector spaces}
\author{Danny Sugrue}
\begin{document}
\begin{abstract}
The Cantor-Bendixson rank of a topological space $X$ is a measure of the complexity of the topology of $X$. The Cantor-Bendixson rank is most interesting when the space is profinite: Hausdorff, compact and totally disconnected. We will see that the injective dimension of the Abelian category of sheaves of $\mathbb{Q}$-modules over a profinite space $X$ is determined by the Cantor-Bendixson rank of $X$.
\end{abstract}
\maketitle
\section{Introduction}
The injective dimension of an object of an Abelian category tells us the minimal number of non-zero terms of any injective resolution of the object. The injective dimension of an Abelian category is the supremum of this value ranging over all objects of the category. This value gives us a bound $n$ for which the groups of any injective resolution of the category at level $k$ are trivial for $k$ bigger than $n$. We are interested in the Abelian category of sheaves of $\mathbb{Q}$-modules over a profinite space $X$ and this paper will show that the injective dimension of this category can be completed simply by looking at the Cantor-Bendixson rank of $X$. Our main interest is in the particular case where $X$ is profinite however the results in this paper hold for more general spaces. One reason why we are especialy interested in profinite spaces is because the concept of the Cantor-Bendixson rank is especially interesting in this setting. This is because there are many interesting examples of profinite spaces which have isolated points which accumulate at limit points in interesting ways. This is best illustrated in two papers by Gartside and Smith \cite{Gartside} and \cite{Gartside1} where the profinite space given by the space of closed subgroups of a profinite group $G$ denoted $SG$ is studied for varying $G$. Furthermore if $G$ is a compact Lie group then $SG/G$ is a profinite space. The main result of the paper is given in the following theorem, see Theorems \ref{ID_CB} and \ref{ID_CB_infty}.

\begin{theorem*}
If $X$ is a space which is scattered and of finite Cantor-Bendixson rank $n$ then the injective dimension of sheaves of $\mathbb{Q}$-modules over $X$ is $n-1$. If $X$ is any space with infinite Cantor-Bendixson rank then the injective dimension is also infinite.
\end{theorem*}

If $X$ has finite Cantor-Bendixson rank and non-empty perfect hull then certain difficulties arise and these are discussed after the main theorem. This discussion leads to Conjecture \ref{Conject} which states that in the remaining case the injective dimension is infinite and this is left for future work. The results presented in this paper apply to any semisimple ring $R$ since every $R$-module is injective (this is the key fact we use about $\mathbb{Q}$-modules). In \cite{SugrueT} the results of this paper are extended to rational $G$-equivariant sheaves, for $G$ a profinite group. This allows us to calculate the injective dimension of an algebraic model for rational $G$-spectra.

The first section of the paper will introduce the concept of the Cantor-Bendixson rank of a Hausdorff space and look at some useful applications of this concept. For the second section we will set up the injective resolutions that we will be using in the final calculations. The final section works towards proving the main theorem and discussing the conjecture dealing with the remaining case.

\section{Cantor-Bendixson Rank}
Given a profinite space $X$, the aim of this paper is to calculate the injective dimension of the category of sheaves of $\mathbb{Q}$-modules over $X$ in terms of the Cantor-Bendixson rank of $X$ denoted by ${\rank}_{CB}(X)$. We begin by defining and stating known properties of the Cantor-Bendixson rank which are known in \cite{Gartside1} and \cite{Gartside}. Recall that an isolated point of a topological space $X$ is a point $x$ which satisfies that $\left\lbrace x\right\rbrace$ is open in $X$.
\begin{definition}
For a topological space $X$ we can define the \textbf{Cantor-Bendixson process} on $X$ where $X^{\prime}$ is the set of all isolated points of $X$:
\begin{enumerate}
\item Let $X^{(0)}=X$ and $X^{(1)}=X\setminus X^{\prime}$.
\item For successor ordinals suppose we have $X^{(\alpha)}$ for an ordinal $\alpha$, we define $X^{(\alpha+1)}=X^{(\alpha)}\setminus {X^{(\alpha)}}^{\prime}$.
\item If $\lambda$ is a limit ordinal we define $X^{(\lambda)}=\underset{\alpha<\lambda}{\colim}X^{(\alpha)}$.
\end{enumerate}
\end{definition}
Every Hausdorff topological space $X$ has a minimal ordinal $\alpha$ such that $X^{(\alpha)}=X^{(\lambda)}$ for all $\lambda\geq \alpha$, see Gartside and Smith \cite[Lemma 2.7]{Gartside}.
\begin{definition}
Let $X$ be a Hausdorff topological space. Then we define the \textbf{Cantor-Bendixson rank} of $X$ denoted $\rank_{CB}(X)$ to be the minimal ordinal $\alpha$ such that $X^{(\alpha)}=X^{(\lambda)}$ for all $\lambda\geq \alpha$.
\end{definition}
There are two ways that the Cantor-Bendixson process can stabilise. The first way is where the process stabilises to give the empty set and the second is where it stabilises at a non-empty space which we call the perfect hull.
\begin{definition}
A compact Hausdorff space $X$ of Cantor-Bendixson rank $\alpha$ is called \textbf{scattered} if the space $X^{(\alpha)}$ obtained by the definition above is equal to the empty set. A topological space $X$ is called \textbf{perfect} if it has no isolated points.
\end{definition}
\begin{remark}
If $X$ is perfect or $X\neq \emptyset$ then $\rank_{CB}(X)=0$.
\end{remark}
\begin{example}\label{padic}
Consider $P=\left\lbrace \frac{1}{n}\mid n\in\mathbb{N}\right\rbrace\bigcup\left\lbrace 0\right\rbrace$ with the subspace topology of $\mathbb{R}$. For any prime number $p$, the $p$-adic integers $\mathbb{Z}_p$ satisfy that $S(\mathbb{Z}_p)$ is homeomorphic to $P$. We know that $\rank_{CB}(P)=1$.
\end{example}
We will see more interesting examples after Proposition \ref{prodcant}. With this definition in mind we can understand the following theorem:
\begin{theorem}[Cantor-Bendixson Theorem]\label{CantThm}
Given a countably based Hausdorff topological space $X$, we can write $X$ as a disjoint union of a scattered subspace and a perfect subspace called the perfect hull of $X$. 
\end{theorem}
Therefore under the assumptions of the above theorem we have that $X=X_S\coprod X_H$ where $X_S$ is the scattered part of $X$ and $X_H$ is the perfect hull.

\begin{proof}
Follows from \cite[Lemma 2.7]{Gartside} and the subsequent paragraph.
\end{proof}
\begin{remark}
In general if $X$ is a Hausdorff space then $X_H$ is always closed and $X_S$ is always open. The assumption that $X$ is countably based in the above theorem allow us to prove that $X_H$ and $X_S$ are a closed-open decomposition of $X$. Without this assumption $X_S$ does not have to be closed.
\end{remark}
If we have a space $X$ with $x\in X_S$ we define the height of $x$ denoted $\text{ht}(X,x)$ to be the ordinal $\kappa$ such that $x\in X^{(\kappa)}$ but $x\notin X^{(\kappa+1)}$. We sometimes denote this by $\rank_{CB}(x)$.

\begin{proposition}\label{prodcant}
Let $X$ be a space with $\rank_{CB}(X)=n+1$ and $Y$ be a space with $\rank_{CB}(Y)=m+1$, where $m,n\in\mathbb{N}$. Then $\rank_{CB}(X\prod Y)=m+n+1=\rank_{CB}(X)+\rank_{CB}(Y)-1$.
\end{proposition}
\begin{proof}
We first prove this in the case where both $X$ and $Y$ are scattered. Let $X_k$ denote the set of isolated points in $X^{(k)}$, and $Y_k$ denote the set of isolated points in $Y^{(k)}$. First note that the isolated points of $\left(X\prod Y\right)$ are given by $X_0\prod Y_0$, and so:
\begin{align*}
\left(X\prod Y\right)^{(1)}=\left(X\prod Y\right)\setminus \left(X_0\prod Y_0\right).
\end{align*}
The isolated points of $\left(X\prod Y\right)^{(1)}$ are equal to the set
\begin{align*}
\left(X_0\prod Y_1\right)\coprod\left(X_1\prod Y_0\right)
\end{align*}
and it therefore follows that:
\begin{align*}
\left(X\prod Y\right)^{(2)}=\left(X\prod Y\right)\setminus \left[(X_0\prod Y_0)\coprod\left(X_0\prod Y_1\right)\coprod \left(X_1\prod Y_0\right)\right].
\end{align*}
This is true since for example if $(x,y)$ were such that $x$ had height greater than or equal to $2$ and $y$ were isolated then the points of the form $(x^{\prime},y)$ would converge to $(x,y)$, where $x^{\prime}$ has height between $1$ and the height of $x$.

\textbf{Claim:} The isolated points in $\left(X\prod Y\right)^{(i)}$ are of the form
\begin{align*}
\underset{p+q=i}{\coprod}\left(X_p\prod Y_q\right)
\end{align*}
where $0\leq p\leq n$ and $0\leq q \leq m$.

We have shown this holds for $i=0$ and $i=1$ so let the above claim be our inductive hypothesis, and suppose it holds for $i$ and that $\gamma$ is an isolated point of $\left(X\prod Y\right)^{(i+1)}$. 

Then if $i$ is even and hence $i+1$ is odd, since all of the points accumulating at $\gamma$ were eliminated in the previous stage of the Cantor-Bendixson process, and by our hypothesis each of these accumulation points which were isolated in $\left(X\prod Y\right)^{(i)}$ belong to some $X_p\prod Y_q$ where $p+q=i$, so $\gamma$ must belong to $X_{p+1}\prod Y_q$ or $X_p \prod Y_{q+1}$. 

In the case where $i$ is odd and $i+1$ is even we have the same possibilities plus the additional possibility where $\gamma$ is in $X_{\frac{i+1}{2}}\prod Y_{\frac{i+1}{2}}$. Therefore we have shown by induction that the isolated points are of the form $\underset{p+q=i+1}{\coprod}\left(X_p\prod Y_q\right)$.

From this we can see that:
\begin{align*}
\left(X\prod Y\right)^{(i)}=\left(X\prod Y\right)\setminus \left[\underset{0\leq k\leq i-1}{\coprod}\left(\underset{p+q=k}{\coprod}X_p\prod Y_q\right)\right]. 
\end{align*}
We then have:
\begin{align*}
\left(X\prod Y\right)^{(n+m-1)}&=\left(X_n\prod Y_{m-1}\right)\coprod \left(X_{n-1}\prod Y_m\right)\\
\left(X\prod Y\right)^{(n+m)}&=\left(X_n\prod Y_m\right)\\
\left(X\prod Y\right)^{(n+m+1)}&=\emptyset.
\end{align*}
This proves that 
\begin{align*}
\rank_{CB}(X\prod Y)=m+n+1=\rank_{CB}(X)+\rank_{CB}(Y)-1. 
\end{align*}
The case where atleast one of $X$ and $Y$ are non-scattered is similar except we observe that we end up with
\begin{align*}
\left(X\prod Y\right)^{(n+m+1)}=\left(X\prod Y\right)_H
\end{align*}
which may not be empty. 
\end{proof}
We now have the following two examples of Cantor-Bendixson rank calculations. 
\begin{example}
From \cite[Proposition 2.5]{Gartside1} we know that if $q_1,q_2,\ldots,q_n$ are a finite collection of distinct primes then there is an isomorphism:
\begin{align*}
S\left(\underset{1\leq i\leq n}{\prod}\mathbb{Z}_{q_i}\right)\cong \underset{1\leq i\leq n}{\prod}S\left(\mathbb{Z}_{q_i}\right)
\end{align*}
By Example \ref{padic} there is a homeomorphism of spaces:
\begin{align*}
S\left(\underset{1\leq i\leq n}{\prod}\mathbb{Z}_{q_i}\right)\cong P^n
\end{align*}
An application of Proposition \ref{prodcant} shows that:
\begin{align*}
\rank_{CB}\left(S\left(\underset{1\leq i\leq n}{\prod}\mathbb{Z}_{q_i}\right)\right)=n+1.
\end{align*}
\end{example}
\begin{example}\label{Exinf}
The space $\underset{n\in\mathbb{N}}{\coprod}{P^n}$ gives an example of a space which has infinite Cantor-Bendixson rank. This is because we can set $x_n$ to be the point in $P^n$ with height equal to $n+1$ and we therefore have a sequence of points with unbounded height. Notice that this space is not profinite since it is not compact.
\end{example}
\section{Injective Resolutions of Sheaves}
In this section we construct an injective resolution of sheaves of $\mathbb{Q}$-modules over a space $X$. We do this by defining the Godement resolution of a sheaf and outlining why this is injective. If $F$ is a sheaf of $\mathbb{Q}$-modules over a space $X$ we let $LF$ denote the sheaf space corresponding to $F$ and $\pi$ denote the local homeomorphism from $LF$ to $X$.
\begin{definition}\label{serr}
Let $F$ be a sheaf of $\mathbb{Q}$-modules over a topological space $X$. Then we define the sheaf $C^0(F)$ on the open sets $U$ by taking $C^0(F)(U)$ to be the collection of serrations, i.e. the set of not necessarily continuous functions $\left\lbrace f:U\rightarrow LF\,\mid \pi\circ f=\id\right\rbrace$ which equates to $\underset{x \in U}{\prod} F_x$.
\end{definition}
Note that every section is a serration so we have a natural inclusion $\delta_0:F\rightarrow C^0(F)$ which is a monomorphism.
\begin{remark}\label{sermapdef}
The map from a sheaf $F$ into $C^0(F)$ is given as follows:
\begin{align*}
F(U)&\rightarrow \underset{y\in U}{\prod}F_y\rightarrow \underset{V\backepsilon \,x}{\colim}\underset{y\in V}{\prod}F_y\\s&\mapsto (s_y)_{y\in U}\mapsto \left((s_y)_{y\in U}\right)_x
\end{align*}
where $U$ is an open neighbourhood of a point $x\in X$ and $\left(-\right)_x$ is the germ at $x$. This induces a map ${\delta_0}_x$ on stalks as follows:
\begin{align*}
F_x&\rightarrow \underset{V\backepsilon x}{\colim}\underset{y\in V}{\prod}F_y\\s_x&\mapsto \left((s_y)_{y\in U}\right)_x 
\end{align*}
We call ${\delta_0}_x$ the serration map and denote it by $S$ throughout to simplify notation.
\end{remark}
Notice that if a map $f$ belongs to the set of serrations in Definition \ref{serr} then $f$ does not have to be continuous. We can now define the Godemont resolution using Definition \ref{serr} and \cite[Page 36 - 37]{Bredon}.
\begin{definition}\label{Godement}
Let $F$ be a sheaf of $\mathbb{Q}$-modules over a topological space $X$. Then as in Definition \ref{serr} we have $C^0(F)$ and a monomorphism $\delta_0:F\rightarrow C^0(F)$.

Consider $\coker\delta_0$, if we replace $F$ in the construction above with $\coker{\delta_0}$ and set $C^1(F)= C^0(\coker{\delta_0})$ from Definition \ref{serr} we will get the following diagram:
\begin{center}
$\xymatrix{0\ar[r]&F\ar[r]^{\delta_0}&C^0(F)\ar[d]\ar[r]^{\delta_1}&C^1(F)\\
&&\coker{\delta_0}\ar[ur]_{{\delta_1}^{\prime}}}$
\end{center}
where ${\delta_1}^{\prime}$ is the monomorphism from $\coker\delta_0$ into $C^1(F)$. We can then continue to build the resolution inductively using this idea. This resolution which we have constructed is called the \textbf{Godement resolution}.
\end{definition}
\begin{remark}\label{skyprod}
Each $C^0(F)$ can be written as $\underset{y\in X}{\prod}\iota_y(F_y)$. To see this if $U\subseteq X$ is open then $C^0(F)(U)=\underset{y\in U}{\prod}F_y$. On the other hand:
\begin{align*}
\left(\underset{y\in X}{\prod}\iota_y(F_y)\right)(U)&=\underset{y\in X}{\prod}\left(\iota_y(F_y)(U)\right)=\underset{y\in U}{\prod}F_y
\end{align*}
\end{remark}
The following lemma relates the Cantor-Bendixson process to the Godement resolution. It shows that the $k^{\text{th}}$ term of the Godement resolution is concentrated over $X^{(k)}$. This will ultimately provide an upper bound for the injective dimension of sheaves.
\begin{lemma}\label{lem0}
Let $X$ be a topological space and $F$ be a sheaf of $\mathbb{Q}$-modules over $X$. Then for every $k\in\mathbb{N}$ we have that $C^k(F)_x=0$ for every $x \in X\setminus X^{(k)}$.
\end{lemma}
\begin{proof}
We prove this using mathematical induction. For $k=0$ we will start by calculating $C^0(F)_x$ when $x$ is isolated. By definition we have: 
\begin{align*}
{C^0(F)}_x=\underset{U\backepsilon\, x}{\colim}\underset{y\in U}{\prod} F_y
\end{align*}
where $U$ ranges across all neighbourhoods of $x$. Since $x$ is isolated it is clear that $\left\lbrace x\right\rbrace$ is the minimal neighbourhood of $x$, so we have that $C^0(F)_x=F_x$ as well as the fact that the monomorphism ${\delta_0}_x$ is an isomorphism. In particular this says that $\coker{\delta_0}_x=0$. It therefore follows that $C^1(F)_x=0$ since 
\begin{align*}
C^1(F)_x=C^0(\coker{\delta_0})_x={\coker{\delta_0}}_x=0
\end{align*}
Suppose ${\coker\delta_{k-1}}_x=0$ and hence $C^k(F)_x=0$ on $X\setminus X^{(k)}$, and take any $x \in X\setminus X^{(k+1)}$ for some $k\in\mathbb{N}$. First observe that:
\begin{align*}
X\setminus X^{(k+1)}\supseteq X\setminus X^{(k)}.
\end{align*}
If it happens that $x \in X\setminus X^{(k)}$ and hence has height less than $k$, then by hypothesis ${\coker\delta_{k-1}}_x=0$ and hence $C^k(F)_x=0$. Therefore since ${\coker{\delta_{k}}}_x$ is a quotient of $C^k(F)_x$ which is zero it follows that ${\coker{\delta_{k}}}_x=0$. Since any point $y$ which accumulates at $x$ has scattered height less than that of $x$, and hence less than $k$, it follows that ${\coker{\delta_{k}}}_y=0$. We can then see immediately that $C^{k+1}(F)_x=\underset{V\backepsilon\, x}{\colim} \underset{y\in V}{\prod}{\coker{\delta_{k}}}_y=0$.

The final case is the one where $x$ is isolated in $X^{(k)}$ and hence the scattered height of $x$ is equal to $k$. This case yields the following diagram:
\begin{center}
$\xymatrix{&C^k(F)\ar[dr]&&C^{k+1}(F)\\\coker{\delta_{k-1}^{\prime}}\ar[ur]^{\delta_{k}^{\prime}}&&\coker{\delta_{k}^{\prime}}\ar[ur]^{\delta_{k+1}^{\prime}}}$
\end{center}
Observe that all of the points $y$ accumulating at $x$ satisfy that the scattered height of $y$ is less than that of $x$ and hence less than $k$. It follows that ${\coker{\delta_{k-1}^{\prime}}}_y=0$ by the inductive hypothesis for every such $y$. Similar to the $k=0$ step above we have: 
\begin{align*}
C^{k}(F)_x=\underset{U\backepsilon\, x}{\colim}\underset{y\in U}{\prod}\,{\coker{\delta_{k-1}}}_y={\coker{\delta_{k-1}}}_x
\end{align*}
and that ${\delta_{k}^{\prime}}_x$ is an isomorphism so ${\coker{\delta_{k}^{\prime}}}_x=0$. All of the points which accumulate at $x$ must belong to $X\setminus X^{(k)}$ and we have already shown that these points $y$ satisfy ${\coker{\delta_{k}^{\prime}}}_y=0$. This information combined proves that $C^{k+1}(F)_x=\underset{U\backepsilon x}{\colim} \underset{y\in U}{\prod}{\coker{\delta_{k}}}_y=0$.
\end{proof}

Recall the following well-known proposition from category theory which will prove useful and is found in \cite[Proposition 2.3.10]{Weibel}.
\begin{proposition}\label{Adjointinj}
If $\left(F,G\right)$ is an adjoint pair where
\begin{align*}
F:\mathfrak{C}\rightarrow \mathfrak{D}\, \text{and}\, G:\mathfrak{D}\rightarrow \mathfrak{C}
\end{align*}
are functors of Abelian categories, satisfying that $F$ preserves monomorphisms then $G$ preserves injective objects.
\end{proposition}
\begin{proposition}\label{Godinj}
If $F$ is a sheaf of $\mathbb{Q}$-modules over $X$ then $C^k(F)$ is injective in the category of sheaves of $\mathbb{Q}$-modules.
\end{proposition}
\begin{proof}
From the inductive way that $C^k(F)$ is defined it is sufficient to prove that $C^0(F)$ is injective. Remark \ref{skyprod} suggests that it is sufficient to prove that each $\iota_x(F_x)$ is injective. 

This follows from Proposition \ref{Adjointinj} applied to the adjoint pair of functors $\left(Ev_x(-),\iota_x(-)\right)$ where $Ev_x(F)=F_x$ for a sheaf $F$. 

Note that the left adjoint preserves monomorphisms since a monomorphism of sheaves is a morphism of sheaves such that the map at each stalk is a monomorphism of $\mathbb{Q}$-modules.
\end{proof}

We next look at a lemma which helps us with our injective dimension calculations since it will ultimately enable us to calculate $\text{Ext}$ groups. 

Recall from Definition \ref{Godement} that if $x\in X$ and $k<\text{ht}(X,x)$ then:
\begin{align*}
{\coker\delta_k}_x=\left[\underset{U\backepsilon\,x}{\colim}\underset{y\in U}{\prod}{\coker\delta_{k-1}}_y\right]/S
\end{align*}
where $S$ is the serration map from Remark \ref{sermapdef}. Explicitly if $a\in\coker{\delta_{k-1}}_x$ we define $(a,\underline{0})_x$ to be the element in $\coker{\delta_k}_x$ which is the germ at $x$ of the family which is $a$ in place $x$ and zero elsewhere. 
\begin{lemma}\label{Homologychar}
Suppose $X$ is a space with $\rank_{CB}(X)=n$ for $n\in \mathbb{N}$ such that $X^{(n)}=\emptyset$. Then for $j\leq n-1$, $x\in X^{(j)}$ and $F$ a sheaf over $X$, we have an isomorphism $\hom(\iota_x(\mathbb{Q}),C^k(F))\cong \coker{\delta_{k-1}}_x$ for $k<j$, and the map:
\begin{align*}
{\delta_{k+1}}_*:\hom(\iota_x(\mathbb{Q}),C^k(F))\rightarrow \hom(\iota_x(\mathbb{Q}),C^{k+1}(F))
\end{align*}
is given by the map:
\begin{align*}
\alpha_{k+1}:\coker{\delta_{k-1}}_x&\rightarrow \coker{\delta_{k}}_x\\a&\mapsto \left(a,\underline{0}\right)_x.
\end{align*}
\end{lemma}
\begin{proof}
Firstly notice that $C^k(F)$ is defined to be $C^0(\coker\delta_{k-1})$, so we begin by proving that $\hom(\iota_x(\mathbb{Q}),C^0(F))\cong F_x$. Observe that $C^0(F)=\underset{y\in X}{\prod}\iota_y(F_y)$ so we can write:
\begin{align*}
\hom(\iota_x(\mathbb{Q}),C^0(F))&=\hom(\iota_x(\mathbb{Q}),\underset{y\in X}{\prod}\iota_y(F_y))=\underset{y\in X}{\prod}\hom(\iota_x(\mathbb{Q}),\iota_y(F_y))\\&=\underset{y\in X}{\prod}\hom(\iota_x(\mathbb{Q})_y,F_y)=F_x.
\end{align*}
In particular if $f\in \hom(\iota_x(\mathbb{Q}),C^k(F))$, then this is determined by a map in $\hom(\iota_x(\mathbb{Q}),\iota_x(\coker{\delta_{k-1}}_x))$. This corresponds to a point $f_x\in \coker{\delta_{k-1}}_x$, so $f$ is given by the element: 
\begin{align*}
[f_x,\underline{0}]_x\in \underset{V\backepsilon\, x}{\colim}\left(\underset{y\in V}{\prod}\coker{\delta_{k-1}}_y\right)=C^0(\coker{\delta_{k-1}})_x,
\end{align*}
with this germ at $x$ of the family taking value $f_x$ in position $x$ and $0$ elsewhere. It follows that ${\delta_{k+1}}_*(f)$ corresponds to $\delta_{k+1}([f_x,\underline{0}]_x)$.

But we therefore have:
\begin{align*}
\delta_{k+1}([f_x,\underline{0}])=\left(\left[\left(f_x,\underline{0}\right)_x\right]^S,\left(\left[\left(f_x,\underline{0}\right)_y\right]^S\right)_{y\in U}\right)_x
\end{align*}
in $\underset{V\backepsilon\, x}{\colim}\left(\underset{y\in V}{\prod}\coker{\delta_{k}}_y\right)=C^0(\coker{\delta_{k}})_x$, where $\left[-\right]^S$ represents the class in the cokernel of the map $S$. This follows from the definition of the maps $\delta$ in Definition \ref{Godement}.

However if $x\neq y$ then since $X$ is Hausdorff there is a neighbourhood of $y$ not containing $x$ so $\left[\left(f_x,\underline{0}\right)_y\right]^S=\left[\left(\underline{0}\right)_y\right]^S$. Therefore $\delta_{k+1}\left(\left(f_x,\underline{0}\right)_x\right)$ can be written as $\left(\left[\left(f_x,\underline{0}\right)_x\right]^S,\left[\left(\underline{0}\right)_y\right]^S\right)_{y\in U}$.

We therefore have that ${\delta_{k+1}}_*$ is defined as follows:
\begin{align*}
{\delta_{k+1}}_*:\hom(\iota_x(\mathbb{Q}),C^k(F))&\rightarrow \hom(\iota_x(\mathbb{Q}),C^{k+1}(F))\\ [f_x,\underline{0}]_x &\mapsto \left(\left[\left(f_x,\underline{0}\right)_x\right]^S,\left[\left(\underline{0}\right)_y\right]^S\right)_{y\in U}
\end{align*} 

Since we are interested in what the maps correspond to as maps between the $x$ components of the products of $C^0(\coker\delta_{k-1})$ and $C^0(\coker\delta_k)$, we observe that it sends $f_x$ to $\left[\left(f_x,\underline{0}\right)_x\right]^S$.
\end{proof}
Now we consider the preceeding lemma for points in $X$ which either have infinite height or belong to the hull of $X$.
\begin{lemma}\label{homoloinf}
Suppose $X$ is a space with infinite Cantor-Bendixson rank. Then for $x\in X^{(n)}$ for any $n\in \mathbb{N}$, and $F$ a sheaf over $X$, we have an isomorphism $\hom(\iota_x(\mathbb{Q}),C^k(F))\cong \coker{\delta_{k-1}}_x$ for $k<n$, and the map:
\begin{align*}
{\delta_{k+1}}_*:\hom(\iota_x(\mathbb{Q}),C^k(F))\rightarrow \hom(\iota_x(\mathbb{Q}),C^{k+1}(F))
\end{align*}
is given by the map:
\begin{align*}
\alpha_{k+1}:\coker{\delta_{k-1}}_x&\rightarrow \coker{\delta_{k}}_x\\a&\mapsto \left(a,\underline{0}\right)_x.
\end{align*}
This also holds for points of infinite height and points in the hull.
\end{lemma}
\begin{proof}
First note that since $X$ has infinite Cantor-Bendixson rank we have that for each $n\in \mathbb{N}$ there exists a point $x_n$ with height $n$. Therefore the result follows from Lemma \ref{Homologychar}. A similar argument also works for points of infinite height. Also the argument in Lemma \ref{Homologychar} also holds for any point in the perfect hull of $X$ for each $k$ since any such point is never removed in the Cantor-Bendixson process.
\end{proof}
The previous two lemmas indicate how the calculations in this paper differ from sheaf cohomology. In this setting we apply the functor $\hom(\iota_x(\mathbb{Q}),-)$ to an injective resolution and this is different from sheaf cohomology where we apply a functor $\hom(A,-)$ for some constant sheaf $A$.
\section{Injective dimension calculation}
We now formally give the definition of the injective dimension of sheaves of $\mathbb{Q}$-modules over a space $X$ as seen in \cite[Definition 4.1.1, Definition 10.5.10]{Weibel}.
\begin{definition}
The \textbf{injective dimension} of a sheaf $F$ over $X$ denoted by $\ID(F)$ is the minimum positive integer $n$ (if it exists) such that there is an injective resolution of the form
\begin{center}
$\xymatrix{0\ar[r]&X\ar[r]^{\epsilon}&I_0\ar[r]^{f_0}&I_1\ar[r]^{f_1}&\ldots\ar[r]^{f_{n-1}} &I_n\ar[r]&0}$.
\end{center}
where $I_j\neq 0$ for $j\leq n$. It is infinite if such a value doesn't exist.
\end{definition}
From \cite[4.1.2]{Weibel} we define the injective dimension of the category of sheaves of $\mathbb{Q}$-modules to be:
\begin{align*}
\sup\left\lbrace \ID(F)\mid\,F\,\in\text{Sheaf}_{\mathbb{Q}}(X)\right\rbrace, 
\end{align*}
where $\text{Sheaf}_{\mathbb{Q}}(X)$ denotes the category of sheaves of $\mathbb{Q}$-modules over $X$. We can now verify that the injective dimension of sheaves of $\mathbb{Q}$-modules is bounded above for a particular class of space.
\begin{proposition}
If $X$ is a scattered space with $\rank_{CB}(X)=n$ for $n\in\mathbb{N}$ then the injective dimension of sheaves of $\mathbb{Q}$-modules over $X$ is bounded above by $n-1$.
\end{proposition}
\begin{proof} 
By Proposition \ref{Godinj} the Godement resolution is an injective resolution. An application of Lemma \ref{lem0} shows that the terms of the Godement resolution are zero after term $n-1$. Therefore the injective dimension of each sheaf is less than or equal to $n-1$.
\end{proof}
In order to get equality it is sufficient to find a particular sheaf $F$ for which $\ID(F)\geq \rank_{CB}(X)-1$. To achieve this we look at \cite[Lemma 4.1.8, Exercise 10.7.2]{Weibel} which says $\ID(F)\leq \rank_{CB}(X)-2$ if and only if $\text{Ext}^{\rank_{CB}(X)-1}(A,F)=0$ for every sheaf $A$. In particular if we can find sheaves $A$ and $F$ such that $\text{Ext}^{\rank_{CB}(X)-1}(A,F)\neq 0$ then we must have that $\ID(F)>\rank_{CB}(X)-2$. This then forces $\ID(F)=\rank_{CB}(X)-1$.

Notice that the concept of an $\text{Ext}$ group in the category of $R$-modules for a ring $R$ stated above in \cite[Chapter 4]{Weibel} holds for general Abelian categories by \cite[Corollary 10.7.5]{Weibel}.

We now work towards verifying the lower bound. We will look at the following lemma which will illustrate that the Godement resolution is non-zero at term $k$ provided $k$ is less than $\rank_{CB}(X)$.

\begin{lemma}\label{godementnon0}
Let $X$ be a non-empty scattered space and $k\in\mathbb{N}$ be less than or equal to $\rank_{CB}(X)$. For each $x\in X^{(k)}$, ${\coker\delta_{k-1}}_x\neq 0$ in the Godement resolution of $c\mathbb{Q}$ the constant sheaf at $\mathbb{Q}$. 

Furthermore if $X$ is any space with a non-empty perfect hull and $k\in\mathbb{N}$, then for each $x\in X^{(k)}$ we have ${\coker\delta_{k-1}}_x\neq 0$ in the Godement resolution of $c\mathbb{Q}$.
\end{lemma}
\begin{proof}
We begin with the scattered case and we will prove this using an induction argument. Since 
\begin{align*}
\underset{U\backepsilon\, x}{\colim}\,c\mathbb{Q}(U)=c\mathbb{Q}_x=\underset{U\backepsilon\,x}{\colim}\,Pc\mathbb{Q}(U)=\mathbb{Q}
\end{align*}
where $Pc\mathbb{Q}$ represents the presheaf, any $q_x\in \mathbb{Q}_x$ is represented by some $q\in \mathbb{Q}$. We therefore have the following diagram by \cite[pages 36-37]{Bredon}:
\begin{align*}
\mathbb{Q}&\rightarrow\underset{y\in U}{\prod}\mathbb{Q}\rightarrow\underset{V\backepsilon\, x}{\colim} \underset{y\in V}{\prod}\mathbb{Q}\\q&\mapsto(q_y)_{y\in U}\mapsto\left((q_y)_{y\in U}\right)_x
\end{align*} 
which induces a map:
\begin{align*}
\mathbb{Q}&\rightarrow \underset{V\backepsilon\, x}{\colim}\underset{y \in V}{\prod}\mathbb{Q}\\q_x&\mapsto \left((q_y)_{y\in U}\right)_x
\end{align*}
We call this map the serration map and denote it by $S$. This is not surjective since we have a point $\left(0_x,\underline{1}\right)_x$ not in the image of $S$. This point is non-zero since if $S(a)=\left[\left(0_x,\underline{1}\right)_x\right]^S$ then $a_x=0$ and so there is an open neighbourhood $U$ of $x$ such that $a_y=0$ for $y\in U$. However the definition of the serration map shows that $a_y=1$ also for $y\neq x$ which is a contradiction. Therefore $\coker{\delta_0}_x\neq 0$ for $x\in X^{(1)}$.

Suppose this holds up to some $n\in\mathbb{N}$ and for any $x\in X^{(n+1)}$. By assumption we have that 
\begin{align*}
0\neq \coker{\delta_n}_x=\underset{U\backepsilon \,x}{\colim}\left(\prod_{y\in U}\coker{\delta_{n-1}}_y\right)/S
\end{align*}
Using the fact that sheafification preserves stalks of presheaves we have a map using \cite[pages 36-37]{Bredon} as follows:
\begin{align*}
\left(\underset{y\in U}{\prod}\coker{\delta_{n-1}}_y\right)/S&\rightarrow \underset{z\in U}{\prod}\coker{\delta_n}_z\rightarrow\underset{V\backepsilon \,x}{\colim}\underset{z\in V}{\prod}\coker{\delta_n}_z\\ \left[(a_y)_{y\in U}\right]^S&\mapsto \left(\left[\left((a_y)_{y\in U}\right]^S\right)_z\right)_{z\in U}\mapsto \left(\left(\left(\left[(a_y)_{y\in U}\right]^S\right)_z\right)_{z\in U}\right)_x 
\end{align*}
which induces a map:
\begin{align*}
\coker{\delta_n}_x&\rightarrow \underset{V\backepsilon\, x}{\colim}\underset{y\in V}{\prod}\coker{\delta_n}_y\\ \left(\left[(a_y)_{y\in U}\right]^S\right)_x&\mapsto \left(\left(\left(\left[(a_y)_{y\in U}\right]^S\right)_z\right)_{z\in U}\right)_x
\end{align*}
Let $U$ be any open neighbourhood of $x$. Then for each $y\in U$ such that $y\in X^{(n)}$ or $X^{(n+1)}$ we can choose $0\neq a_y\in \coker{\delta_{n-1}}_y$ by the inductive hypothesis. Set $s^y=\left(0_y,a_z\right)_{z\in U\setminus\left\lbrace y\right\rbrace}$, then $\left[\left(s^y\right)_y\right]^S\neq 0$ in $\coker{\delta_n}_y$ and we denote this by $b_y$. This is shown to be non-zero by following a similar argument to that seen earlier in this proof. We can therefore consider:
\begin{align*}
\left(\left[\left(0_x,b_y\right)_{y\in U\setminus\left\lbrace x\right\rbrace}\right]^S\right)_x
\end{align*}
which is not in the image of the serration map so $\coker{\delta_{n+1}}_x\neq 0$. This is also seen by referring to the previous argument seen earlier in this proof.

Note if $\rank_{CB}(X)$ is infinite then for each $k\in \mathbb{N}$ we have that each $X^{(k)}$ has isolated points to remove, so this is true for every $k$. If $\rank_{CB}(X)=n$ and $X^{(n)}=\emptyset$ then $X^{(n-1)}$ is discrete and therefore satisfies that $C^n(F)=0$ by Lemma \ref{lem0}. It follows that the argument therefore only results in non-zero stalks for $k\leq n-1$. If $x$ belongs to the perfect hull of $X$ then this argument also holds for each $k\in\mathbb{N}$ since the hull is contained in each $X^{(k)}$.
\end{proof}

We now use the above calculations to verify the injective dimension of sheaves using the Cantor-Bendixson dimension.
\begin{theorem}\label{ID_CB}
Suppose $X$ is a space with $\rank_{CB}(X)=n$ such that $X^{(n)}=\emptyset$. Then the category of sheaves over $X$ has injective dimension equal to $n-1$.
\end{theorem}
\begin{proof}
To see this we need to find an object in the category of sheaves over $X$ so that $\ID(X)=n-1$, we will show that $c\mathbb{Q}$ satisfies $\id(c\mathbb{Q})=n-1$. Firstly by Lemmas \ref{lem0} we know that $\ID(c\mathbb{Q})\leq n-1$ since the Godement resolution gives an injective resolution of the form:
\begin{align*}
\xymatrix{0\ar[r]&c\mathbb{Q}\ar[r]^{\delta_0}&I^0\ar[r]^{\delta_1}&\ldots\ar[r]^{\delta_{n-2}}&I^{n-2}\ar[r]^{\delta_{n-1}}&I^{n-1}\ar[r]&0}
\end{align*} 
From Lemma \ref{godementnon0} we know that each $I_j\neq 0$. We will show that the $Ext^{n-1}\left(\iota_x(\mathbb{Q}),c\mathbb{Q}\right)$ group calculated by the above injective resolution is non-zero. Let $x$ be an element of $X$ with $\text{ht}(X,x)=n-1$ (any point of $X$ with maximal height). 

We apply the functor $\Hom(\iota_x(\mathbb{Q}),-)$ and forget the $c\mathbb{Q}$ term to get:
\begin{align*}
\xymatrix{\Hom(\iota_x(\mathbb{Q}),I^0)\ar[r]^-{{\delta_1}_*}&\Hom(\iota_x(\mathbb{Q}),I^1)\ar[r]^-{{\delta_2}_*}&\ldots\ar[r]^-{{\delta_{n-2}}_*}&\Hom(\iota_x(\mathbb{Q}),I^{n-2})\ar[d]^{{\delta_{n-2}}_*}\\&&0&\Hom(\iota_x(\mathbb{Q}),I^{n-1})\ar[l]}
\end{align*}
which we can no longer assume to be exact. This is equal to the following sequence:
\begin{align*}
\xymatrix{\mathbb{Q}\ar[r]^-{\alpha_1}&\coker{\delta_0}_x\ar[r]^-{\alpha_2}&\ldots\ar[r]^-{\alpha_{n-3}}&\coker{\delta_{n-4}}_x\ar[d]^{\alpha_{n-2}}\\&0&\coker{\delta_{n-2}}_x\ar[l]^{\alpha_{n}}&\coker{\delta_{n-3}}_x\ar[l]^{\alpha_{n-1}}}
\end{align*}
We want to show that $\text{Ext}^{(n-1)}\left(\iota_x(\mathbb{Q}),c\mathbb{Q}\right)=\ker\alpha_n/\im\alpha_{n-1}\neq 0$ and $\text{Ext}^n\left(\iota_x(\mathbb{Q}),c\mathbb{Q}\right)=\ker \alpha_{n+1}/\im \alpha_n=0$. It is clear that $\text{Ext}^n\left(\iota_x(\mathbb{Q}),c\mathbb{Q}\right)$ is $0$. For the other we need to prove that the map:
\begin{align*}
\alpha_{n-1}:\coker {\delta_{n-3}}_x&\rightarrow \coker {\delta_{n-2}}_x\\a&\mapsto\left(a,\underline{0}\right)_x
\end{align*}
is not surjective. This is done in a similar fashion to Propositions \ref{godementnon0}.

For any open neighbourhood $U$ of $x$ there are infinitely many points $z$ of $U$ such that $z\in X^{(n-2)}$ and $\coker{\delta_{n-3}}_z\neq 0$ so we can choose such a point $a_z$ for each $z$. Consider:
\begin{align*}
a=\left[\left((0_x,a_z)_{z\in U\setminus\left\lbrace x\right\rbrace}\right)_x\right]^S\in \coker{\delta_{n-2}}_x. 
\end{align*}
If $t_x\in \coker{\delta_{n-3}}_x$ is in the preimage of $a$ with respect to $\alpha_{n-1}$ we would have:
\begin{align*}
\left[\left(t_x,0\right)_x\right]^S=\alpha_{n-1}(t_x)=\left[\left((0_x,a_z)_{z\in U\setminus\left\lbrace x\right\rbrace}\right)_x\right]^S
\end{align*}
so $t_x=0$ which implies that $\left[\left((0_x,a_z)_{z\in U\setminus\left\lbrace x\right\rbrace}\right)_x\right]^S=0$. But this cannot be the case since there are infinitely many $z$ satisfying that $a_z\neq \underline{0}$ by construction, so we have a contradiction and $\alpha_{n-1}$ cannot be surjective.
\end{proof}
We now deal with the case where the Cantor-Bendixson dimension is infinite.
\begin{theorem}\label{ID_CB_infty}
If $X$ is a space with infinite Cantor-Bendixson rank, then the injective dimension of sheaves of $\mathbb{Q}$-modules over $X$ is infinite.
\end{theorem}
\begin{proof}
Since the Cantor-Bendixson rank of $X$ is infinite there exists a sequence of points $x_n$ each having height $n$. As a consequence of Theorem \ref{ID_CB} for each $x_n$ we know that $\text{Ext}^n\left(\iota_{x_n}(\mathbb{Q}),c\mathbb{Q}\right)\neq 0$ and this happens for each $n$ since we do not have a maximal height. This proves the result. 
\end{proof}
We are left to deal with the case where $X$ is not scattered but has finite Cantor-Bendixson rank. In the above cases when we resolve with respect to $\iota_x\left(\mathbb{Q}\right)$ the resolved sequence becomes zero after $\text{ht}(x)-1$ so the kernel of the map $\alpha_{\text{ht}(x)+1}$ is $\coker{\delta_{\text{ht}(x)-1}}_x$ which is advantageous since we can choose any point of $\coker{\delta_{\text{ht}(x)-1}}$ not in the image of $\alpha_{\text{ht}(x)}$. This changes when we are working in the case where $X$ has a perfect hull.

We know that the following Godement resolution is infinite:
\begin{align*}
\xymatrix{0\ar[r]&c\mathbb{Q}\ar[r]^{\delta_0}&I^0\ar[r]^{\delta_1}&\ldots\ar[r]^{\delta_{n-1}}&I^{n-1}\ar[r]^{\delta_{n}}&I^n\ar[r]^{\delta_n}&\ldots}
\end{align*} 
Therefore after resolving like above for a point $x$ in the perfect hull we obtain the following infinite sequence:
\begin{align*}
\xymatrix{0\ar[r]&\mathbb{Q}\ar[r]^-{\alpha_1}&\coker{\delta_0}_x\ar[r]^-{\alpha_2}&\ldots\ar[r]^-{\alpha_{n-2}}&\coker{\delta_{n-3}}_x\ar[d]^{\alpha_{n-1}}\\&&\ldots&\coker{\delta_{n-1}}_x\ar[l]^{\alpha_{n+1}}&\coker{\delta_{n-2}}_x\ar[l]^{\alpha_{n}}}
\end{align*}
The important thing to notice is that since this is non-zero at infinitely many places, when calculating the group $Ext^n$ we can't just chose any representative of $\coker{\delta_{n-1}}_x$ since the kernel is not everything.

In order to choose something in the kernel we need to adjust our argument above, namely instead of choosing a representative $(0_x,s^y)_{y\in U\setminus\left\lbrace x\right\rbrace}$ with $0\neq s^y\in\coker{\delta_{n-2}}_y$ arbitrary, we need $(s^y)_{y\in U\setminus\left\lbrace x\right\rbrace}$ to be determined by a section $s$ over $\coker{\delta_{n-2}}$. That is we want each $s^y$ to be of the form $s_y$ for that section $s$, and such that each open neighbourhood $U$ of $x$ contains infinitely many $y$ such that $s^y\neq 0$.

Recalling a fact from sheaf theory that a section $s$ over an open neighbourhood $U$ of $x$ has germ $s_x=0$ if and only if $s$ restricts to some smaller neighbourhood to give the zero section. Also recall that we can build a section in $\coker\delta_{n-2}(U)$ by considering $\underset{y\in U}{\prod}\coker\delta_{n-3}/\text{Serrate}$. This means that if we can construct the family $\left[(a^y)_{y\in U\setminus\left\lbrace x\right\rbrace}\right]_S$ to be an alternating family where infinitely many $a^y\neq 0$ in ${\coker\delta_{n-3}}_y$ and infinitely many do equal zero, then we may have a suitable section $s$ to proceed with the proof. This approach needs the following condition to proceed:

If $a^y=0$ then then any neighbourhood $U$ of $y$ contains infinitely many points $z$ such that $a^z\neq 0$.

If we can construct given any net converging to $x$, two term-wise disjoint subnets then we can do the above construction to show that the injective dimension of sheaves in this case is infinite, provided the sequence is set up to satisfy the condition. With this in mind we have the following conjecture.
\begin{conjecture}\label{Conject}
If $X$ has finite Cantor-Bendixson rank and non-empty perfect hull then the injective dimension of sheaves of $\mathbb{Q}$-modules over $X$ is infinite.
\end{conjecture}

We now look at examples of spaces and the application of the result relating injective dimension of sheaves of $\mathbb{Q}$-modules over $X$ to the Cantor-Bendixson rank of $X$. Our primary interest is spaces of closed subgroups, see \cite{SugrueT} for more details. 
\begin{example}
If $G$ is a discrete group then $SG$ is a finite discrete space and hence has Cantor-Bendixson dimension $1$. Therefore Theorem \ref{ID_CB} implies that the injective dimension of sheaves of $\mathbb{Q}$-modules over $SG$ is $0$.
\end{example} 
The above example works equally for any discrete space.
\begin{example}
If $G=\mathbb{Z}_p$ for $p$ any prime number, then $S\left(\mathbb{Z}_p\right)$ is homeomorphic to the space $P$ as defined in Example \ref{padic}. For more detail see the paragraph proceeding Definition 3.2 \cite{BarZp} and \cite{Gartside}. This space has Cantor-Bendixson dimension $2$ hence the category of sheaves of $\mathbb{Q}$-modules over $S\left(\mathbb{Z}_p\right)$ has injective dimension $1$ by Theorem \ref{ID_CB}. 
\end{example}
\begin{example}
Consider distinct primes $p_1,p_2,\ldots,p_n$. We have a profinite group $\underset{1\leq i\leq n}{\prod}\mathbb{Z}_{p_i}$ with corresponding profinite space $S\left(\underset{1\leq i\leq n}{\prod}\mathbb{Z}_{p_i}\right)$. This space is homeomorphic to $P^n$ by \cite[Proposition 2.5]{Gartside}. Then by Proposition \ref{prodcant} we have that $\dim_{CB}(P^n)=n+1$ and $\left(P^n\right)^{(n+1)}=\emptyset$. We can now apply Theorem \ref{ID_CB} to deduce that the injective dimension of sheaves over $S\left(\underset{1\leq i\leq n}{\prod}\mathbb{Z}_{p_i}\right)$ is exactly $n$. 
\end{example}
The following example demonstrates a particular case where the category of sheaves have infinite injective dimension.
\begin{example}
In Example \ref{Exinf} we observed that the space $\underset{n\in\mathbb{N}}{\coprod}P^n$ has infinite Cantor-Bendixson rank. Therefore an application of Theorem \ref{ID_CB_infty} shows that the injective dimension of sheaves of $\mathbb{Q}$-modules over this space is also infinite.
\end{example} 
\begin{remark}
If $G$ is a profinite group which is an inverse limit over a countable diagram then $SG$ is a second countable space. It follows from Theorem \ref{CantThm} that $SG$ is topologically the disjoint union of the scattered part of $SG$ and the perfect hull of $SG$. This in particular means that the injective dimension of sheaves over $SG$ is equal to the maximum of the injective dimension of sheaves of the scattered part of $SG$ and that of the perfect hull of $SG$.  
\end{remark}
Furthermore if we consider Conjecture \ref{Conject} we can see the possible implications.
\begin{example}
The profinite completion of $\mathbb{Z}$ is defined to be:
\begin{align*}
\hat{Z}=\underset{p}{\prod}\mathbb{Z}_p
\end{align*}
where the product runs over the collection of prime numbers $p$. This is a profinite group under the product topology and we can see that $S\left(\hat{\mathbb{Z}}\right)$ is perfect. If proven to be correct, Conjecture \ref{Conject} would imply that the injective dimension of sheaves over this space is infinite.
\end{example}
Another important example of a space is defined in \cite[Definition 2.8]{Gartside}, and this construction is similar to the Cantor space.
\begin{definition}\label{Pel}
Let $F_0=P_0=[0,1]$, the closed unit interval. We set $F_1=F_0\setminus (\frac{1}{3},\frac{2}{3})$ and $B_1=F_1\bigcup \left\lbrace\frac{1}{2}\right\rbrace$. That is, to form $F_1$ we remove the middle third of the interval of $F_0$ and to form $B_1$ we reinsert the midpoint of the deleted interval to $F_1$.

Given $F_{i-1}$ we define $F_i$ by deleting the middle third intervals of the remaining segments of $F_{i-1}$ and we define $B_i$ by reinserting midpoints of the deleted intervals to $F_i$. We set $F=\underset{n\in\mathbb{N}}{\bigcap}F_n$ and $B=\underset{n\in\mathbb{N}}{\bigcap}B_n$.
\end{definition}
The main focus of \cite{Gartside} is on proving that the algebraic structure of a profinite group $G$ can tell us about $SG$. Specifically, throughout \cite{Gartside} there are many assumptions on the algebraic structure of $G$ which lead to the conclusion that $SG$ is homeomorphic to $B$ from Definition \ref{Pel}. 
\begin{example}
Consider the spaces $B$ and $F$ defined in Definition \ref{Pel}. From \cite{Gartside} we know that the space $B$ has perfect hull given by the Cantor space $F$ and that $\rank_{CB}(B)=1$. If Conjecture \ref{Conject} were true then it would follow that the injective dimension of sheaves over $B$ is infinite.
\end{example}
\thanks{I would like to thank my PhD supervisor Dr David Barnes for his excellent guidance and support throughout many interesting discussions about this work. I would also like to thank Dr Martin Mathieu for his very helpful advice about this paper. Finally I would like to thank my family and friends for all of their moral support which has made this possible.}

\bibliography{refpap}
\bibliographystyle{alpha}
\end{document}